\documentclass[journal,twoside,web]{ieeecolor}
\usepackage{generic} 
\bstctlcite{bstctl:etal, bstctl:nodash, bstctl:simpurl}
\def\BibTeX{{\rm B\kern-.05em{\sc i\kern-.025em b}\kern-.08em
    T\kern-.1667em\lower.7ex\hbox{E}\kern-.125emX}}

\usepackage{textcomp}
\usepackage{lcsys}
    
\usepackage[left=54pt, right=54pt,bottom=54pt, top=54pt]{geometry}

\usepackage{amsmath,mathrsfs,amsfonts,amssymb,graphicx,epsfig}
 \usepackage{amsthm}
\usepackage{subcaption}
\usepackage{color,multirow,rotating,tcolorbox}
\usepackage{algorithm,algpseudocode,algorithmicx}
\usepackage{cite,url,framed,bm,balance,nicematrix,physics}
\usepackage{stmaryrd,mathtools}

\newtheorem{theorem}{Theorem}

\newtheorem{proposition}{Proposition}
\newtheorem{remark}{Remark}

\newtheorem{mytheorem}{Theorem}[]      

\usepackage{tikz}
\usetikzlibrary{calc,trees,positioning,arrows,chains,shapes.geometric,decorations.pathreplacing,decorations.pathmorphing,shapes, matrix,shapes.symbols}

\newcommand{\DampedNewton}{%
  \raisebox{2pt}{\tikz{
    \draw[line width=1pt] (0,0) -- (5mm,0);
    \node[diamond, fill, inner sep=1.3pt] at (0.25,0) {};
  }}%
}

\newcommand{\OlikerPrussner}{%
  \protect\raisebox{2pt}{\protect\tikz{
    \protect\draw[black, dashed, line width=1pt] (0,0) -- (5mm,0);
    \protect\filldraw[fill=black, draw=black] (2.5mm,0) circle (1.5pt);
  }}%
}

\usepackage[breaklinks]{hyperref}

\makeatletter
\newcommand{\pushright}[1]{\ifmeasuring@#1\else\omit\hfill$\displaystyle#1$\fi\ignorespaces}

\newcommand{\dotminus}{\mathbin{\text{\@dotminus}}}

\newcommand{\@dotminus}{%
  \ooalign{\hidewidth\raise1ex\hbox{.}\hidewidth\cr$\m@th-$\cr}%
}

\allowdisplaybreaks

\title{\LARGE\textbf{Control Laguerre Tessellation:\\Semi-discrete Optimal Transport Over Control Systems
}}

\author{Ripon C. Sarker, Abhishek Halder
\thanks{Ripon C. Sarker and Abhishek Halder are with the Department of Aerospace Engineering, Iowa State University, Ames, IA 50011, USA, {\tt\footnotesize{\{rcsarker,ahalder\}@iastate.edu}}.}
\thanks{This research was partially supported by NSF award 2111688.}
}

\begin{document}
\bstctlcite{IEEEexample:BSTcontrol}
\maketitle
\thispagestyle{empty}
\pagestyle{empty}

\begin{abstract}
We study the optimal transport of optimally controlled agents from a compactly supported absolutely continuous source to a discrete target measure. The ground cost for the transport is induced by the optimal cost of the agents' motion. When this ground cost satisfies the twist condition, the optimal transport map is given almost everywhere in terms of a Laguerre tessellation of the state space. We refer to this control-theoretic generalization of Laguerre tessellation as Control Laguerre Tessellation (CLT), and illustrate it for two ground costs induced by linear controlled agents with minimum energy and minimum time objectives. 
\end{abstract}


\section{Introduction}\label{sec:Intro}
\IEEEPARstart{C}{onsider} a large population of indistinguishable active agents, i.e., identical control systems, with absolutely continuous normalized\footnote{Here $\int_{\mathcal{X}}\differential\mu = 1$.} population measure $\mu$ supported over a compact state space $\mathcal{X}\subseteq\mathbb{R}^{n}$. Suppose we want to optimally transport this population to a finitely supported target measure $\nu = \sum_{i=1}^{r}\nu_i\delta_{\bm{y}_{i}}$ with fixed $r\in\mathbb{N}_{\geq 2}$, where $\delta_{\bm{y}_i}$ denotes the Dirac delta at $\bm{y}_i\in\mathbb{R}^{n}$ $\forall i\in\llbracket r\rrbracket := \{1,\hdots,r\}$, and $(\nu_1,\hdots,\nu_r)\in\Delta^{r-1}\,\text{(standard simplex)}\,\subset\mathbb{R}^{r}_{\geq 0}$. We assume that the target states are distinct, i.e., $\bm{y}_i \neq \bm{y}_{j}$ for $i\neq j$. The vector\footnote{We use the boldfaced $\bm{\nu}$ for a vector on the standard simplex, and the unboldfaced $\nu$ for the corresponding discrete measure.} $\bm{\nu}:=(\nu_1,\hdots,\nu_r)$ can be interpreted as either normalized capacities or relative importance of the target states $\bm{y}_1, \hdots, \bm{y}_r$. 

The large population of active agents comprising the source measure $\mu$ may represent micron-sized chiplets controlled by electric field for micro-assembly \cite{edwards2014controlling,nodozi2023controlled,matei20232d}, or magnetic nanoparticles for targeted drug delivery \cite{zhang2024nanorobot,snezhko2011magnetic,wang2021trends}. In economic resource allocation, $\mu$ may represent a population (e.g., people, children in a city) and $\nu$ may represent localized resource (e.g., coffee shops, elementary schools) with given capacities. Such applications naturally lead to the \emph{semi-discrete optimal transport (SDOT) problem} \cite[Ch. 5]{gabriel2019computational}:
\begin{align}
\underset{T:\mathcal{X}\rightarrow\{\bm{y}_i\}_{i=1}^{r}\mid T_{\#}\mu = \nu}{\text{minimize}} \quad \int_{\mathcal{X}}c(\bm{x},T(\bm{x}))\differential\mu
\label{SDOT}
\end{align}
where $T_{\#}\mu$ denotes the pushforward of $\mu$ via the transport map $T$, and $c:\mathcal{X}\times\{\bm{y}_i\}_{i=1}^{r}\mapsto\mathbb{R}_{\geq 0}$ is a known ground cost induced by the active agents. The triple $\mu,\nu,c$ comprise the problem data for \eqref{SDOT}. Fig. \ref{fig:IntroFigure} shows an SDOT instance with $n=2$, $r=5$.

For $\bm{x}\in\mathcal{X},i\in\llbracket r\rrbracket$, the $c(\bm{x},\bm{y}_i)$ quantifies the cost of transporting unit amount of mass from the source state $\bm{x}$ to the target state $\bm{y}_i$. The term ``semi-discrete" refers to that the source measure $\mu$ is absolutely continuous w.r.t. the Lebesgue measure on $\mathcal{X}$ while the target measure $\nu$ is discrete (weighted sum of Dirac masses). SDOT has found applications in fluid dynamics \cite{gallouet2018lagrangian,bourne2026semi}, machine learning \cite{de2012blue,houdard2023generative}, mean field games \cite{sarrazin2022lagrangian}, and robotic coverage control \cite{inoue2020optimal,napolitano2026optimal}. 

The purpose of this work is to explore the solution of \eqref{SDOT} when $c$ is induced by an optimal control cost. Our motivation comes from applications such as micro-assembly and drug delivery mentioned earlier, where optimality of the individual agent's controlled trajectories are desired in addition to the collective transport guarantees. Specifically, we consider the individual agents to be identical control systems with initial states distributed in the support of $\mu$. The agents move to minimize certain performance objective (e.g., energy, time), and incur ground cost via their optimal motion. We will formalize the problem in Sec. \ref{sec:ProbFormulation}. 

\begin{figure}
\centering
\includegraphics[width=0.65\linewidth]{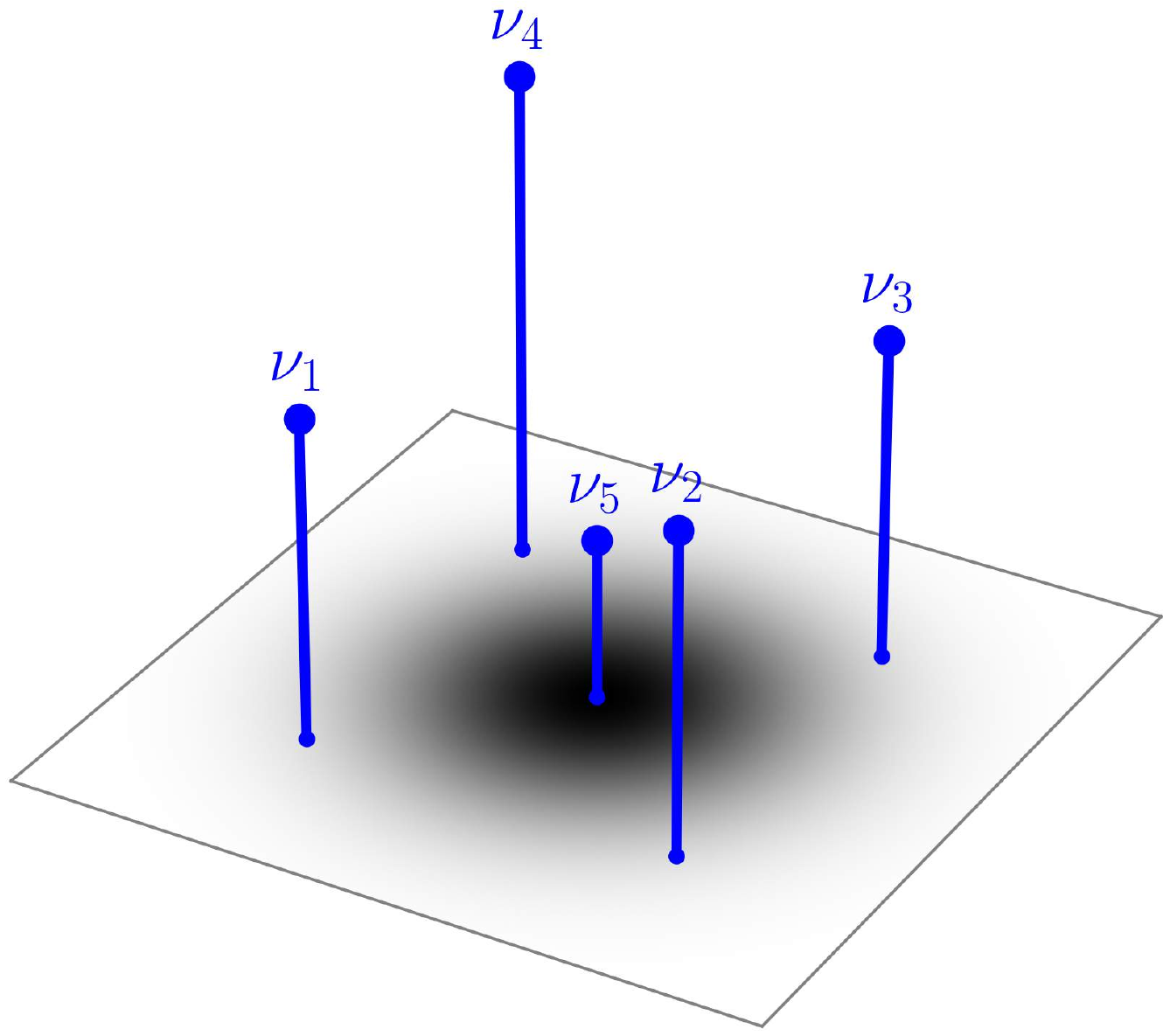}
\caption{{\small{An instance of SDOT for $n=2, r=5$. The source measure $\mu$ (dark = high, light = low values) is supported on $\mathcal{X}=[-1,1]^2$, and $\bm{\nu}=(\nu_1,\hdots,\nu_5) = (0.2, 0.2, 0.2, 0.3, 0.1)$.}}
}
\label{fig:IntroFigure}
\vspace*{-0.2in}
\end{figure}

If the ground cost $c$ satisfies \emph{twist condition} \cite[Def. 1.16]{santambrogio2015optimal}:
\begin{itemize}
\item $c$ is differentiable $\forall \bm{x}\in\mathcal{X}$, and $\bm{y}_i \mapsto \nabla_{\bm{x}}c(\bm{x},\bm{y}_i)\in \mathrm{T}_{\bm{x}}^{*}\mathcal{X}$ (cotangent space) is injective $\forall i\in\llbracket r\rrbracket, \bm{x}\in\mathcal{X}$,
\end{itemize}
$\mu$-a.e. then the minimizer in \eqref{SDOT} a.k.a. the \emph{optimal transport map} $T^{\mathrm{opt}}$ has a simple geometric characterization in terms of the \emph{Laguerre cell} associated to $\bm{y}_i$, $i\in\llbracket r\rrbracket$,
\begin{align}
{\mathrm{Lag}}_{i}(\bm{\psi}) :=\{
\bm{x} \in \mathcal{X}
\mid c(\bm{x},\bm{y}_i) + \psi_i
\le
&c(\bm{x},\bm{y}_j) + \psi_j \nonumber\\ 
&\forall j\in\llbracket r\rrbracket\setminus\{i\}\},
\label{defLaguerreCell}
\end{align}
defined in terms of certain \emph{dual potential} or \emph{weight} vector $\bm{\psi}=(\psi_1,\dots,\psi_r)\in\mathbb{R}^{r}$. Specifically, the following result is well-known \cite[Thm. 40, Prop. 37]{merigot2021optimal}.

\begin{proposition}\label{Prop:Topt}
Consider the semi-discrete optimal transport problem \eqref{SDOT} for a given ground cost $c:\mathcal{X}\times\{\bm{y}_i\}_{i=1}^{r}\mapsto\mathbb{R}_{\geq 0}$ satisfying the twist condition. Then
\begin{itemize}
    \item $\exists T^{\mathrm{opt}}$ minimizing \eqref{SDOT} that is unique $\mu$-almost everywhere (a.e.),
    \item $\exists\bm{\psi}\in\mathbb{R}^{r}$ such that $\mathcal{X}=\sqcup_{i=1}^{r}{\mathrm{Lag}}_{i}(\bm{\psi})$ is a Laguerre tessellation with\footnote{Here ${\mathrm{vol}}_{n}$ denotes the $n$-dimensional Lebesgue volume.} ${\mathrm{vol}}_{n }\!\left({\mathrm{Lag}}_{i}(\bm{\psi}) \cap {\mathrm{Lag}}_{j}(\bm{\psi})\!\right)\! =\! 0\forall i\neq j$, 
    \item $T^{\mathrm{opt}}$ is piecewise constant, and is given $\mu$-a.e. as
\begin{align}
T^{\mathrm{opt}}(\bm{x}) = \bm{y}_i \; \text{if}\;\bm{x}\in{\mathrm{interior}}\!\left({\mathrm{Lag}}_{i}(\bm{\psi})\!\right)\!\forall i\in\llbracket r\rrbracket.
\label{Topt}    
\end{align}
\end{itemize}
\end{proposition}
\noindent Here $\bm{x}$ being in the interior of ${\mathrm{Lag}}_{i}(\bm{\psi})$ refers to a strict inequality in \eqref{defLaguerreCell}. Equality in \eqref{defLaguerreCell} occurs when $\bm{x}$ is at the boundary of two adjacent Laguerre cells (which has ${\mathrm{vol}}_{n}$ measure zero).

Define the map $\bm{G}:\mathbb{R}^{r}\mapsto\Delta^{r-1}$ with components$$G_{i}(\bm{\psi}):=\int_{\mathrm{Lag}_{i}(\bm{\psi})}\!\differential\mu\quad\forall i\in\llbracket r\rrbracket.$$Combining \eqref{Topt} with the constraint $T^{\mathrm{opt}}_{\#}\mu = \nu$ yields a system of $r$ nonlinear equations to solve for the unknown $\bm{\psi}\in\mathbb{R}^{r}$:
\begin{align}
\bm{G}(\bm{\psi})=\bm{\nu}.
\label{GpsiEqualsnu} 
\end{align}
Equation \eqref{GpsiEqualsnu} is the \emph{discrete Monge-Ampère equation} \cite{meyron2019initialization}, and can be seen as direct consequence of Kantorovich duality\footnote{Under the twist condition on $c$, the Kantorovich dual functional$$\mathcal{K}(\bm{\psi}):=\sum_{i=1}^{r}\int_{{\mathrm{Lag}}_{i}(\bm{\psi})}(c(\bm{x},\bm{y}_i)+\psi_i)\differential\mu - \sum_{i=1}^{r}\nu_i\psi_i$$ is concave and $\mathcal{C}^2$ in $\bm{\psi}$, and achieves unconstrained maximum that equals the constrained minimum value in \eqref{SDOT}. The unconstrained maximizer solves $\nabla_{\bm{\psi}}\mathcal{K} = G(\bm{\psi}) - \bm{\nu} = \bm{0}$, which is precisely \eqref{GpsiEqualsnu}.}. Given $\mu$ and $\bm{\nu}$, the solution for \eqref{GpsiEqualsnu} is unique up to an additive constant $k\in\mathbb{R}$ because $\psi_{i} \mapsto \psi_{i} + k$ $\forall i\in\llbracket r\rrbracket$, leaves \eqref{defLaguerreCell}, and thus the Laguerre tessellation invariant. Two standard numerical methods to solve \eqref{GpsiEqualsnu} are the damped Newton algorithm \cite{kitagawa2019convergence} and the Oliker-Prussner coordinate descent \cite[Sec. 4.2]{merigot2021optimal}, both have known convergence guarantees \cite[Sec. 4.2, 4.3]{merigot2021optimal}. Our numerics in Sec. \ref{sec:Numerical} will use these two.

\begin{remark}\label{Remark:DependenceOnGroundCost}
Equation \eqref{GpsiEqualsnu} and its solution depend on the ground cost $c$ because ${\mathrm{Lag}}_{i}(\bm{\psi})$ depends on $c$, see \eqref{defLaguerreCell}.
\end{remark}

\noindent\textbf{Related works.} The Laguerre tessellations generalize the well-known \cite{aurenhammer2013voronoi} Voronoi tessellations: the latter correspond to $\bm{\psi}\equiv\bm{0}$ in \eqref{defLaguerreCell}. There is a large literature \cite{aurenhammer1998minkowski,levy2015numerical,de2019differentiation} with numerical toolbox \cite{pysdot} on SDOT for squared Euclidean ground cost $c(\bm{x},\bm{y})=\|\bm{x}-\bm{y}\|_2^2$, in which case, the Laguerre tessellation is referred to as the \emph{power diagram} \cite[Ch. 6.2]{aurenhammer2013voronoi}. Another well-studied case is un-squared Euclidean ground cost $c(\bm{x},\bm{y})=\|\bm{x}-\bm{y}\|_2$, in which case, the tessellation is known as the \emph{Apollonius diagram} \cite[Ch. 5]{gabriel2019computational}. 

SDOT for other ground costs are relatively less investigated except for $p$-norm costs for $p\in(1,\infty)$ \cite{dieci2019boundary}, for positive combinations of such $p$-norms \cite{dieci2025solving}, and for logarithmic cost on the sphere for reflector design \cite{de2017numerical}. SDOT for optimal control induced ground costs, as studied here, has not appeared before.

\noindent\textbf{Contributions and organization.} 
\begin{itemize}
    \item We propose (Sec. \ref{sec:ProbFormulation}) a generalization of the SDOT problem \eqref{SDOT} and its solution for ground costs induced by optimal control problems with prescribed endpoints.

    \item We detail two instances of the proposed formulation: minimum energy (Sec. \ref{sec:MinEnergy}) and minimum time (Sec. \ref{sec:MinTime}) ground costs for variants of controlled linear dynamics, and prove both admit the generalized solution structure. 

    \item Numerical results (Sec. \ref{sec:Numerical}) illustrate the developments. 
\end{itemize}


\section{Problem Formulation}\label{sec:ProbFormulation}

We consider a variant of the SDOT problem \eqref{SDOT} where the ground cost $c(\bm{x},\bm{y})$ solves an optimal control problem:
\begin{align}
c\left(\bm{x},\bm{y}\right) = &\underset{\bm{u}_t\in\mathcal{U}}{\text{minimum}} \int_{0}^{t_{\mathrm{f}}}L\left(t,\bm{x}_t,\bm{u}_t\right)\:\differential t\nonumber\\
&\text{subject to}\quad\dot{\bm{x}}_t=\bm{f}\left(t,\bm{x}_t,\bm{u}_t\right),\nonumber\\
&\qquad\qquad\quad\bm{x}_{0} =\bm{x}, \quad\bm{x}_{t_{\mathrm{f}}}=\bm{y}.
\label{cGeneric}    
\end{align}
The final time $t_{\mathrm{f}}$ could be fixed or free. At time $t\in[0,t_{\mathrm{f}}]$, the state $\bm{x}_t\in\mathbb{R}^{n}$, the control $\bm{u}_t \in\mathcal{U}\subseteq\mathcal{C}\left([0,t_{\mathrm{f}}];\mathbb{R}^{m}\right)$, the endpoints $\bm{x},\bm{y}\in\mathbb{R}^{n}$ are fixed, and $\bm{x}\neq\bm{y}$. In \eqref{cGeneric}, the vector field $\bm{f}$ and the Lagrangian $L$ are assumed to be sufficiently regular (e.g., $\bm{f}$ being Lipschitz and controllable, $L$ being coercive and strictly convex in $\bm{u}_t$) to guarantee existence-uniqueness of the minimum $c(\bm{x},\bm{y})$.

Notice that when $t_{\mathrm{f}}=1$ fixed, $L = \|\bm{u}_t\|_2^2$, $\bm{f}=\bm{u}_t$, then \eqref{cGeneric} equals $c_{\mathrm{SqEuclidean}}(\bm{x},\bm{y})=\Vert\bm{x}-\bm{y}\Vert_2^2$, the squared Euclidean distance, which is the prototypical SDOT example where Proposition \ref{Prop:Topt} holds and the tessellation of $\mathcal{X}$ is the power diagram. It is natural to expect that for sufficiently regular $\bm{f},L$ in \eqref{cGeneric}, the solution of \eqref{SDOT} will induce a control-theoretic generalization of the Laguerre tessellation, referred hereafter as the \emph{control Laguerre tessellation (CLT)\footnote{Then, the CLT will not only determine the optimal source-to-target
assignment, but will also provide optimality guarantees on the trajectories realizing the population
evolution.}}. Note, however, that the twist condition is not guaranteed by the regularity of $\bm{f},L$ alone and, needs to be verified on a case-by-case basis.

We next illustrate the same for two particular instances of \eqref{cGeneric}, viz. minimum energy (Sec. \ref{sec:MinEnergy}) and minimum time (Sec. \ref{sec:MinTime}) ground costs for variants of controlled linear dynamics.


\section{Minimum Energy SDOT over A Linear Control System}\label{sec:MinEnergy}
The minimum energy guidance of a linear time varying control system from $\bm{x}\in\mathbb{R}^{n}$ to $\bm{y}\in\mathbb{R}^{n}$ over fixed time horizon $[0,1]$ incurs the cost
\begin{align}
c_{\mathrm{MinEnergy}}\left(\bm{x},\bm{y}\right) := &\underset{\bm{u}_t\in\mathcal{U}}{\text{minimum}} \int_{0}^{1}\Vert\bm{u}_t\Vert_2^2\:\differential t\nonumber\\
&\text{subject to}\quad\dot{\bm{x}}_t=\bm{A}_t\bm{x}_t + \bm{B}_t\bm{u}_t,\nonumber\\
&\qquad\qquad\quad\bm{x}_{0} =\bm{x}, \quad\bm{x}_{1}=\bm{y},
\label{cLinearMinEnergy}    
\end{align}
where the matricial trajectory pair $\left(\bm{A}_t,\bm{B}_{t}\right)\in\mathbb{R}^{n\times n}\times\mathbb{R}^{n\times m}$ is bounded and continuous w.r.t. $t\in[0,1]$, and is uniformly controllable in the sense the controllability Gramian $\bm{M}_{ts}$ in any non-empty $[s,t]\subseteq[0,1]$ remains positive definite, i.e.,
\begin{align}
\boldsymbol{M}_{t s}:=\int_s^t \boldsymbol{\Phi}_{t \tau} \boldsymbol{B}_{\tau} \boldsymbol{B}_{\tau}^{\top} \boldsymbol{\Phi}_{t \tau}^{\top} \:\differential\tau \succ\bm{0},\quad 0\leq s < t \leq 1,
\label{DefCtrbGramian}
\end{align}
where $\bm{\Phi}_{ts}$ is the associated state transition matrix. In \eqref{cLinearMinEnergy}, the feasible set $\mathcal{U}$ is given by
\begin{align}
\!\!\mathcal{U}\!:=\!\{\bm{u}_{t}\in\mathcal{C}\left([0,1];\mathbb{R}^{m}\right) \!\mid \!\!\int_{0}^{1}\!\!\!\Vert\bm{u}_t\Vert_2^2\differential t < \infty\, \forall t\in[0,1] \}.
\label{defU}    
\end{align}
The explicit expression for \eqref{cLinearMinEnergy} is well-known \cite[p. 194]{lee1967foundations}, \cite[Sec. III-B]{teter2023contraction}:
\begin{align}
c_{\mathrm{MinEnergy}}(\bm{x},\bm{y}) = \left(\bm{\Phi}_{10}\bm{x}-\bm{y}\right)^{\top}\bm{M}_{10}^{-1}\left(\bm{\Phi}_{10}\bm{x}-\bm{y}\right). \label{cLinearMinEnergyExplicit} 
\end{align}
In the special case $\left(\bm{A}_t,\bm{B}_{t}\right)\equiv\left(\bm{0},\bm{I}\right)$ $\forall t\in[0,1]$, we have $\bm{\Phi}_{10}=\bm{M}_{10}=\bm{I}$, and \eqref{cLinearMinEnergyExplicit} reduces to $c_{\mathrm{SqEuclidean}}(\bm{x},\bm{y})$. 

For this ground cost, our main result is the following.

\begin{theorem}\label{thm:MinEnergyTwist}
Consider the SDOT problem \eqref{SDOT} with ground cost $c_{\mathrm{MinEnergy}}$ given by \eqref{cLinearMinEnergy}, absolutely continuous source measure $\mu$ supported on compact $\mathcal{X}\subset\mathbb{R}^{n}$, and a finitely supported target measure $\nu = \sum_{i=1}^{r}\nu_i\delta_{\bm{y}_{i}}$ where $\bm{y}_i\in\mathbb{R}^{n}$ $\forall i\in\llbracket r\rrbracket := \{1,\hdots,r\}$, and $(\nu_1,\hdots,\nu_r)\in\Delta^{r-1}$. Then
\begin{itemize}
    \item[(i)] $\exists$ $T^{\mathrm{opt}}$ that is unique $\mu$-a.e., and is given by \eqref{Topt} for some $\bm{\psi}\in\mathbb{R}^{r}$ solving \eqref{GpsiEqualsnu},
    \item[(ii)] $\mathcal{X}=\sqcup_{i=1}^{r}{\mathrm{Lag}}_{i}(\bm{\psi})$ is a convex polyhedral tessellation. 
\end{itemize}
\end{theorem}
\begin{proof}
(i) From \eqref{cLinearMinEnergyExplicit}, $c_{\mathrm{MinEnergy}}$ is quadratic and thus differentiable w.r.t. $\bm{x}$. Using the chain rule,
$$\nabla_{\bm{x}}c_{\mathrm{MinEnergy}}(\bm{x},\bm{y}) = 2\bm{\Phi}_{10}^{\top}\bm{M}_{10}^{-1}\left(\bm{\Phi}_{10}\bm{x}-\bm{y}\right),$$
which is affine in $\bm{y}$. In particular, the affine map $$\bm{y}\mapsto 2\bm{\Phi}_{10}^{\top}\bm{M}_{10}^{-1}\left(\bm{\Phi}_{10}\bm{x}-\bm{y}\right)$$ is injective iff the matrix $\bm{\Phi}_{10}^{\top}\bm{M}_{10}^{-1}$ has trivial nullspace. This is indeed the case because the state transition matrix $\bm{\Phi}_{10}$ is nonsingular, and $\bm{M}_{10}^{-1}$ is positive definite. Thus, $c_{\mathrm{MinEnergy}}$ satisfies the twist condition. Hence Proposition \ref{Prop:Topt} applies and claim (i) follows. 

(ii) From \eqref{cLinearMinEnergyExplicit},
$c_{\mathrm{MinEnergy}}(\bm{x},\bm{y})=\|\bm{M}_{10}^{-1/2}\bm{\Phi}_{10}\bm{x}-\bm{M}_{10}^{-1/2}\bm{y}\|_2^2$. Substituting this expression for ground cost in \eqref{defLaguerreCell}, expanding the squares and regrouping terms, we find
\begin{align}
&{\mathrm{Lag}}_{i}(\bm{\psi}) = \{\bm{x}\in\mathbb{R}^{n}\mid\langle 2\bm{\Phi}_{10}^{\top}\bm{M}_{10}^{-1}\left(\bm{y}_j - \bm{y}_{i}\right),\bm{x}\rangle \nonumber\\
&\leq \bm{y}_j^{\top}\bm{M}_{10}^{-1}\bm{y}_j - \bm{y}_i^{\top}\bm{M}_{10}^{-1}\bm{y}_i + \psi_j - \psi_i \quad\forall j\in\llbracket r\rrbracket\setminus\{i\}\}.
\label{LagiLTVMinEnergy}    
\end{align}
Notice that \eqref{LagiLTVMinEnergy} is of the form $\bigcap_{j\neq i}\{\bm{x}\in\mathbb{R}^{n}\mid \langle\bm{a}_{ij},\bm{x}\rangle\leq b_{ij}\}$, so ${\mathrm{Lag}}_{i}(\bm{\psi})$ is an intersection of $r-1$ halfspaces, i.e., a convex polyhedron $\forall i\in\llbracket r\rrbracket$. That $\mathcal{X}=\sqcup_{i=1}^{r}{\mathrm{Lag}}_{i}(\bm{\psi})$ is a tessellation follows from Proposition \ref{Prop:Topt}. This completes the proof of claim (ii).
\end{proof}
\begin{remark}\label{Remark:SplCasePowerDiagm}
Recalling that in the special case $\left(\bm{A}_t,\bm{B}_{t}\right)\equiv\left(\bm{0},\bm{I}\right)$ $\forall t\in[0,1]$, $\bm{\Phi}_{10}=\bm{M}_{10}=\bm{I}$, the expression \eqref{LagiLTVMinEnergy} for the $i$th Laguerre cell simplifies, and $\mathcal{X}=\sqcup_{i=1}^{r}{\mathrm{Lag}}_{i}(\bm{\psi})$ reduces to the power diagram. This is expected since then $\ell$ becomes the identity map, and $c_{\mathrm{MinEnergy}} = c_{\mathrm{SqEuclidean}}$.
\end{remark}


\section{Minimum Time SDOT}\label{sec:MinTime}
Inspired by \cite{bakolas2010zermelo}, we consider minimum time ground cost between a source location $\bm{x}\in\mathbb{R}^{n}$ and a target location $\bm{y}\in\mathbb{R}^{n}$, $\bm{x}\neq\bm{y}$, given by
\begin{align}
\!\!\!\!c_{\mathrm{MinTime}}\left(\bm{x},\bm{y}\right) := &\underset{\bm{u}_{t}\in\mathcal{U}_{1}}{\text{minimum}} \int_{0}^{t_{\mathrm{f}}}\differential t\nonumber\\
&\text{subject to}\;\dot{\bm{x}}_t=\bm{u}_t + \bm{w}_t,\nonumber\\
&\qquad\qquad\bm{x}_{0} =\bm{x}, \quad\bm{x}_{t_{\mathrm{f}}}=\bm{y}\neq\bm{x},
\label{cZermeloMinTime}    
\end{align}
where the feasible set
\begin{align}
\mathcal{U}_{1}:=\{\bm{u}_t \in\mathcal{C}\left([0,\infty);\mathbb{R}^{m}\right) \mid \|\bm{u}_t\|_2 \leq 1\:\forall t\geq 0\},
\end{align}
and $\bm{w}_t$ is a \emph{known} time-varying exogenous vector field. We assume that $\bm{w}_t$ is continuous and $\|\bm{w}_t\|_2<1\:\forall t\geq 0$. The CLT for $c_{\mathrm{MinTime}}$ discussed next generalizes the Zermelo-Voronoi diagram \cite{bakolas2010zermelo} in that the latter is the $\bm{\psi}\equiv\bm{0}$ special case of CLT.

If we consider the state $\bm{x}_t$ to be position of a vehicle in $\mathbb{R}^2$ or $\mathbb{R}^{3}$, then we may interpret $\bm{w}_t$ as wind speed or river current vector field. The ground cost in \eqref{cZermeloMinTime} computes the minimum time $t_{\mathrm{f}}$ to steer the state $\bm{x}\in\mathbb{R}^{n}$ at $t=0$ to $\bm{y}\in\mathbb{R}^{n}$ via bounded control from $\mathcal{U}_{1}$ in the presence of known additive exogenous input $\bm{w}_t$.

Problem \eqref{cZermeloMinTime} and its generalizations have a long history in the optimal control literature, see e.g., \cite{zermelo1931navigationsproblem,kelley1962guidance}. However, an SDOT formulation where individual agents solve instances of \eqref{cZermeloMinTime}, is new.

Unlike $c_{\mathrm{MinEnergy}}$ in Sec. \ref{sec:MinEnergy}, no general formula is available for $c_{\mathrm{MinTime}}$. We establish the following.

\begin{theorem}
Consider the SDOT problem \eqref{SDOT} with ground cost $c_{\mathrm{MinTime}}$ given by \eqref{cZermeloMinTime} where $\bm{w}_t$ is continuous and $\|\bm{w}_t\|_2<1\:\forall t\geq 0$, absolutely continuous source measure $\mu$ supported on compact $\mathcal{X}\subset\mathbb{R}^{n}$, and a finitely supported target measure $\nu = \sum_{i=1}^{r}\nu_i\delta_{\bm{y}_{i}}$ where $\bm{y}_i\in\mathbb{R}^{n}$ $\forall i\in\llbracket r\rrbracket := \{1,\hdots,r\}$, and $(\nu_1,\hdots,\nu_r)\in\Delta^{r-1}$. Let $\bar{T}:=\max_{(\bm{x},\bm{y})\in\mathcal{X}\times\{\bm{y}_{i}\}_{i=1}^{r}}c_{\mathrm{MinTime}}(\bm{x},\bm{y})$. For every $\bm{x} \in \mathcal{X}$ and $\bm{y} \in
\{\bm{y}_i\}_{i=1}^{r}$, if the resulting $\bm{u}^{\mathrm{opt}}$ in \eqref{cZermeloMinTime} satisfies the projection $\langle\bm{u}^{\mathrm{opt}},\int_0^t\bm{w}_{\tau}\differential\tau\rangle$
is strictly convex or strictly concave in $t\in[0,\bar{T}]$ in the $\mu$-a.e. sense, then
\begin{itemize}
    \item[(i)] $\exists$ $T^{\mathrm{opt}}$ that is unique $\mu$-a.e., and is given by \eqref{Topt} for some $\bm{\psi}\in\mathbb{R}^{r}$ solving \eqref{GpsiEqualsnu},
    \item[(ii)] $\mathcal{X}=\sqcup_{i=1}^{r}{\mathrm{Lag}}_{i}(\bm{\psi})$ is a tessellation. 
\end{itemize}
\end{theorem}
\begin{proof}
We will show that $c_{\mathrm{MinTime}}$ satisfies the twist condition. Defining new state coordinate $\bm{z}_t := \bm{x}_t - \int_{0}^{t}\bm{w}_{\tau}\differential\tau$, we have $\dot{\bm{z}}_t = \bm{u}_t$ and 
\begin{align}
\|\bm{z}_{t_{\mathrm{f}}} - \bm{z}_0\|_2 = \bigg\Vert\int_{0}^{t_{\mathrm{f}}}\bm{u}_t\differential t\bigg\Vert_2 \leq \int_{0}^{t_{\mathrm{f}}}\Vert\bm{u}_t\Vert_2 \differential t \leq t_{\mathrm{f}},
\label{zCoordInequality}    
\end{align}
where the last inequality follows from the constraint $\bm{u}_t \in\mathcal{U}_1$. Thus, $t_{\mathrm{f}}$ is minimal when equality occurs in \eqref{zCoordInequality}, which is achieved by the unit vector $\bm{u}^{\mathrm{opt}}=(\bm{z}_{t_{\mathrm{f}}} - \bm{z}_0)/\|\bm{z}_{t_{\mathrm{f}}} - \bm{z}_0\|_2$. Then the optimal value
$t_{\mathrm{f}}^{\mathrm{opt}} = \|\bm{z}_{t_{\mathrm{f}}} - \bm{z}_0\|_{2}$. Since $c_{\mathrm{MinTime}}\equiv t_{\mathrm{f}}^{\mathrm{opt}}$, $\bm{z}_0 = \bm{x}$, $\bm{z}_{t_{\mathrm{f}}} = \bm{y}-\int_{0}^{t_{\mathrm{f}}}\bm{w}_{\tau}\differential\tau$, so $c_{\mathrm{MinTime}}$ solves 
\begin{align}
c_{\mathrm{MinTime}} = \Vert\bm{y}-\int_{0}^{c_{\mathrm{MinTime}}}\bm{w}_{\tau}\differential\tau - \bm{x}\Vert_2.
\label{ImplicitEqn}
\end{align}
For convenience, we square both sides of \eqref{ImplicitEqn} and let
\begin{align}
&F(c_{\mathrm{MinTime}},\bm{x})\nonumber\\
&:= c_{\mathrm{MinTime}}^2 - \Vert\bm{y}-\int_{0}^{c_{\mathrm{MinTime}}}\bm{w}_{\tau}\differential\tau - \bm{x}\Vert_2^2 = 0,
\label{defFimplicit}
\end{align}
treating $\bm{y}$ as a fixed parameter with $\bm{y}\neq\bm{x}$ (per assumption). Now our idea is to apply the implicit function theorem \cite[Thm. 9.28]{rudin1976principles} to \eqref{defFimplicit} for establishing differentiability of $c_{\mathrm{MinTime}}$ w.r.t. $\bm{x}$. This requires us to verify two things: 
\begin{itemize}
    \item $F$ is $\mathcal{C}^{1}$ w.r.t. both $c_{\mathrm{MinTime}},\bm{x}$,
    \item $\frac{\partial F}{\partial c_{\mathrm{MinTime}}}\neq 0\:\forall\bm{x}\neq\bm{y}$.
\end{itemize}

For the $\mathcal{C}^{1}$ smoothness, it suffices to note that by fundamental theorem of calculus, the term $\int_{0}^{c_{\mathrm{MinTime}}}\bm{w}_{\tau}\differential\tau$ is $\mathcal{C}^{1}$ w.r.t. $c_{\mathrm{MinTime}}$ because $\bm{w}_t$ is continuous w.r.t. $t\geq 0$, and that $\|\cdot\|_2^2$ is a smooth function of its arguments. 

To show $\frac{\partial F}{\partial c_{\mathrm{MinTime}}}$ does not vanish, using \eqref{defFimplicit} we compute
$$\frac{\partial F}{\partial c_{\mathrm{MinTime}}}= 2c_{\mathrm{MinTime}} - 2\langle\bm{y}-\int_{0}^{c_{\mathrm{MinTime}}}\!\!\!\!\bm{w}_{\tau}\differential\tau - \bm{x},-\bm{w}_{c_{\mathrm{MinTime}}}\rangle.$$Recalling $\bm{u}^{\mathrm{opt}}=(\bm{y}-\int_{0}^{c_{\mathrm{MinTime}}}\!\!\bm{w}_{\tau}\differential\tau - \bm{x})/c_{\mathrm{MinTime}}$, this expression simplifies to
\begin{align}
\frac{\partial F}{\partial c_{\mathrm{MinTime}}} = 2c_{\mathrm{MinTime}}(1 + \langle\bm{u}^{\mathrm{opt}},\bm{w}_{c_{\mathrm{MinTime}}}\rangle).
\label{partialFpartialc}
\end{align}
Note that $c_{\mathrm{MinTime}}\neq 0\forall\bm{x}\neq\bm{y}$. By the Cauchy-Schwarz inequality, $\langle\bm{u}^{\mathrm{opt}},\bm{w}_{c_{\mathrm{MinTime}}}\rangle \geq -\|\bm{u}^{\mathrm{opt}}\|_2 \|\bm{w}_{c_{\mathrm{MinTime}}}\|_2$. Because $\bm{u}^{\mathrm{opt}}$ is a unit vector, $1+\langle\bm{u}^{\mathrm{opt}},\bm{w}_{c_{\mathrm{MinTime}}}\rangle\geq 1 - \|\bm{w}_{c_{\mathrm{MinTime}}}\|_2$. Also, $\|\bm{w}_t\|_2 < 1\:\forall t\geq 0$ implies $1 - \|\bm{w}_{c_{\mathrm{MinTime}}}\|_2 > 0$. Therefore,
\begin{align}
1 + \langle\bm{u}^{\mathrm{opt}},\bm{w}_{c_{\mathrm{MinTime}}}\rangle >0.
\label{SignOfOnePlusInnerProduct}\end{align}
So \eqref{partialFpartialc} equals twice of a nonzero term times a positive term, hence is nonzero $\forall\bm{x}\neq\bm{y}$. 

Having verified both conditions needed to invoke the implicit function theorem, differentiability of $c_{\mathrm{MinTime}}$ w.r.t. $\bm{x}$ follows from the same.

We next show that the mapping $\bm{y}\mapsto\nabla_{\bm{x}}c_{\mathrm{MinTime}}(\bm{x},\bm{y})$ is injective for fixed $\bm{x}$. Applying Leibniz rule, \eqref{ImplicitEqn} gives
\begin{align}
\nabla_{\bm{x}}c_{\mathrm{MinTime}} = -\frac{\bm{u}^{\mathrm{opt}}}{1 + \langle \bm{u}^{\mathrm{opt}},\bm{w}_{c_{\mathrm{MinTime}}}\rangle}.
\label{gradxOfc}    
\end{align}
For fixed $\bm{x}$, suppose there exist two terminal states $\bm{y}_1,\bm{y}_2$ such that $\nabla_{\bm{x}}c_{\mathrm{MinTime}}(\bm{x},\bm{y}_1)=\nabla_{\bm{x}}c_{\mathrm{MinTime}}(\bm{x},\bm{y}_2)=\bm{g}$ (the common gradient vector). Our goal is to prove that $\bm{y}_1 = \bm{y}_2$. Suppose the corresponding minimum times are $c_1,c_2$ where $c_1\neq c_2$, and the respective optimal controls $\bm{u}^{\mathrm{opt}}_{c_{1}},\bm{u}^{\mathrm{opt}}_{c_{2}}$. From \eqref{SignOfOnePlusInnerProduct} and \eqref{gradxOfc}, the optimal control is negative scalar times the common gradient vector $\bm{g}$. However, the optimal controls being unit vectors, we must have $\bm{u}^{\mathrm{opt}}_{c_{1}}=\bm{u}^{\mathrm{opt}}_{c_{2}}=\bm{u}^{\mathrm{opt}}=-\bm{g}/\|\bm{g}\|_2$. Then \eqref{gradxOfc} yields 
\begin{align}
\langle \bm{u}^{\mathrm{opt}},\bm{w}_{c_{1}}-\bm{w}_{c_{2}}\rangle=0.
\label{ZeroInnerProduct}
\end{align}

As $\bm{w}_t$ is continuous, $\langle\bm{u}^{\mathrm{opt}},\int_0^t\bm{w}_{\tau}\differential\tau\rangle$ is $\mathcal{C}^{1}$ in time $t$. So strict convexity or concavity of $\langle\bm{u}^{\mathrm{opt}},\int_0^t\bm{w}_{\tau}\differential\tau\rangle$ w.r.t. $t\in[0,\bar{T}]$ implies that its derivative, namely $\langle\bm{u}^{\mathrm{opt}},\bm{w}_{t}\rangle$, is strictly increasing or decreasing in $t\in[0,\bar{T}]$. Therefore, $c_{1}\neq c_2$ implies $\langle \bm{u}^{\mathrm{opt}},\bm{w}_{c_{1}}-\bm{w}_{c_{2}}\rangle\neq 0$, contradicting \eqref{ZeroInnerProduct}. Hence $c_1 = c_2$, and $\bm{u}^{\mathrm{opt}}_{c_{1}}=\bm{u}^{\mathrm{opt}}_{c_{2}}$ yields
$$\frac{\bm{y}_1-\int_{0}^{c_1}\!\!\bm{w}_{\tau}\differential\tau - \bm{x}}{c_1}=\frac{\bm{y}_2-\int_{0}^{c_1}\!\!\bm{w}_{\tau}\differential\tau - \bm{x}}{c_1}\:\Rightarrow\:\bm{y}_1=\bm{y}_2,$$
establishing the desired injectivity.

Since $c_{\mathrm{MinTime}}$ is differentiable w.r.t. $\bm{x}$, and $\bm{y}\mapsto\nabla_{\bm{x}}c_{\mathrm{MinTime}}(\bm{x},\bm{y})$ is injective, so $c_{\mathrm{MinTime}}$ satisfies the twist condition. Then Proposition \ref{Prop:Topt} applies and (i)-(ii) follows.
\end{proof}
\begin{remark}
The projection condition amounts to $\langle\bm{u}^{\mathrm{opt}},\bm{w}_{t}\rangle$ not taking the same value twice, and was used in the injectivity proof. If $\bm{x}_t$ denotes position, then this condition says that the angle between the corrective speed $\bm{u}^{\mathrm{opt}}$ and the exogeneous wind speed or river current $\bm{w}_{t}$, is monotone.  
\end{remark}
\begin{remark}\label{Remark:NonconvexCLT}
Unlike the case for $c_{\mathrm{MinEnergy}}$ in Theorem \ref{thm:MinEnergyTwist}, the tessellation induced by $c_{\mathrm{MinTime}}$ can be nonconvex. E.g., see the numerical results in Table \ref{tab:picturegrid}, third column. 
\end{remark}


\section{Numerical Experiments}\label{sec:Numerical}
In this Section, we illustrate and compare the CLTs for three ground costs, viz. the squared Euclidean ground cost $c_{\mathrm{SqEuclidean}}$, the minimum energy ground cost $c_{\mathrm{MinEnergy}}$ in \eqref{cLinearMinEnergy}, and the minimum time ground cost $c_{\mathrm{MinTime}}$ in \eqref{cZermeloMinTime}. All numerics were performed using Python version \texttt{3.13} on a \texttt{13th Gen Intel(R) Core(TM) i9-13900} processor with 2.00 GHz clock speed and 32 GB RAM.

\noindent\textbf{Source and target measures.} To compare the CLTs with different ground costs, we fix the source measure
$\mu=\mathcal{N}\left(\bm{0}_{2\times 1},0.13 \bm{I}_2\right)$ re-normalized over the domain $\mathcal{X} = [-1,1]^2$. We fix the target measure
$$\nu = \displaystyle\sum_{i=1}^{5}\nu_i\delta_{\bm{y}_{i}}, \quad (\nu_1,\hdots,\nu_5) = (0.2, 0.2, 0.2, 0.3, 0.1),$$
where the vectors $\bm{y}_1,\hdots,\bm{y}_5\in\mathbb{R}^2$ are the first to fifth columns of the $2\times 5$ matrix
$$\begin{bmatrix}
-0.5 & 0.5 & 0.5 & -0.5 & 0\\
-0.5 & -0.5 & 0.5 & 0.5 & 0
\end{bmatrix}.$$
These $\mu,\nu$ are shown in Fig. \ref{fig:IntroFigure}.

\begin{figure}
\includegraphics[width=\linewidth]{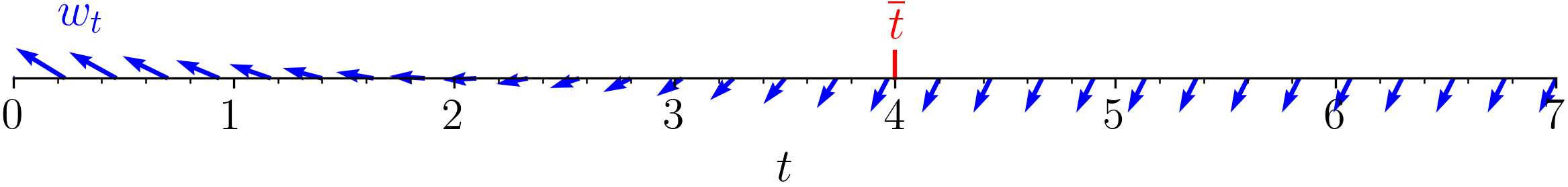}
\caption{{\small{The exogeneous input $\bm{w}_t$ in \eqref{wprofile} over time $t\in[0,7]$.}}}
\label{fig:w}
\vspace*{-0.2in}
\end{figure}

\noindent{\textbf{Control System Parameters for $c_{\mathrm{MinEnergy}}$.}} For minimum energy ground cost  \eqref{cLinearMinEnergy}, we use the double integrator, i.e.,
$$\bm{A}_t \equiv \begin{pmatrix}
0 & 1\\
0 & 0
\end{pmatrix}, \quad \bm{B}_t \equiv\begin{pmatrix}0\\1\end{pmatrix} \quad\forall t\in[0,1],$$
and hence \cite[Theorem 1(ii)]{haddad2020prediction},
$$\bm{\Phi}_{10} = \begin{pmatrix}
1 & 1\\
0 & 1
\end{pmatrix}, \quad \bm{M}_{10}^{-1} =  \begin{pmatrix}
12 & -6\\
-6 & 4
\end{pmatrix}.$$
From \eqref{cLinearMinEnergyExplicit}, the computation of $c_{\mathrm{MinEnergy}}$ is then analytical.

\begin{table*}[t!]
            \centering
\caption{{\small{Numerical results for the simulation set up detailed in Sec. \ref{sec:Numerical}. The three columns are for the three ground costs: $c_{\mathrm{SqEuclidean}}$, $c_{\mathrm{MinEnergy}}$, $c_{\mathrm{MinTime}}$. \emph{First row}: CLTs $\sqcup_{i=1}^{5}R_i$ where $R_i := {\mathrm{Lag}}_i(\bm{\psi})$ for the targets $\{\bm{y}_i\}_{i=1}^{5}$  with ground cost contours, \emph{second row}: comparison of CLTs from the damped Newton and Oliker-Prussner algorithms with source $\mu$ as background color (dark=high, light=low values) and target $\nu$ as red filled circles (large=high, small=low values), \emph{third row}: the evolution of error $\|\bm{G}-\bm{\nu}\|_{\infty}$ over the iteration index, \emph{fourth row}: convergence of the components of the dual potential or weight $\bm{\psi}\in\mathbb{R}^{5}$ (\protect\DampedNewton~damped Newton, \protect\OlikerPrussner~Oliker-Prussner).}}
}
            \label{tab:picturegrid}
                \begin{tabular}{ccc}
                $c_{\mathrm{SqEuclidean}}$ & $c_{\mathrm{MinEnergy}}$ & $c_{\mathrm{MinTime}}$\\
\includegraphics[width=0.32\linewidth]{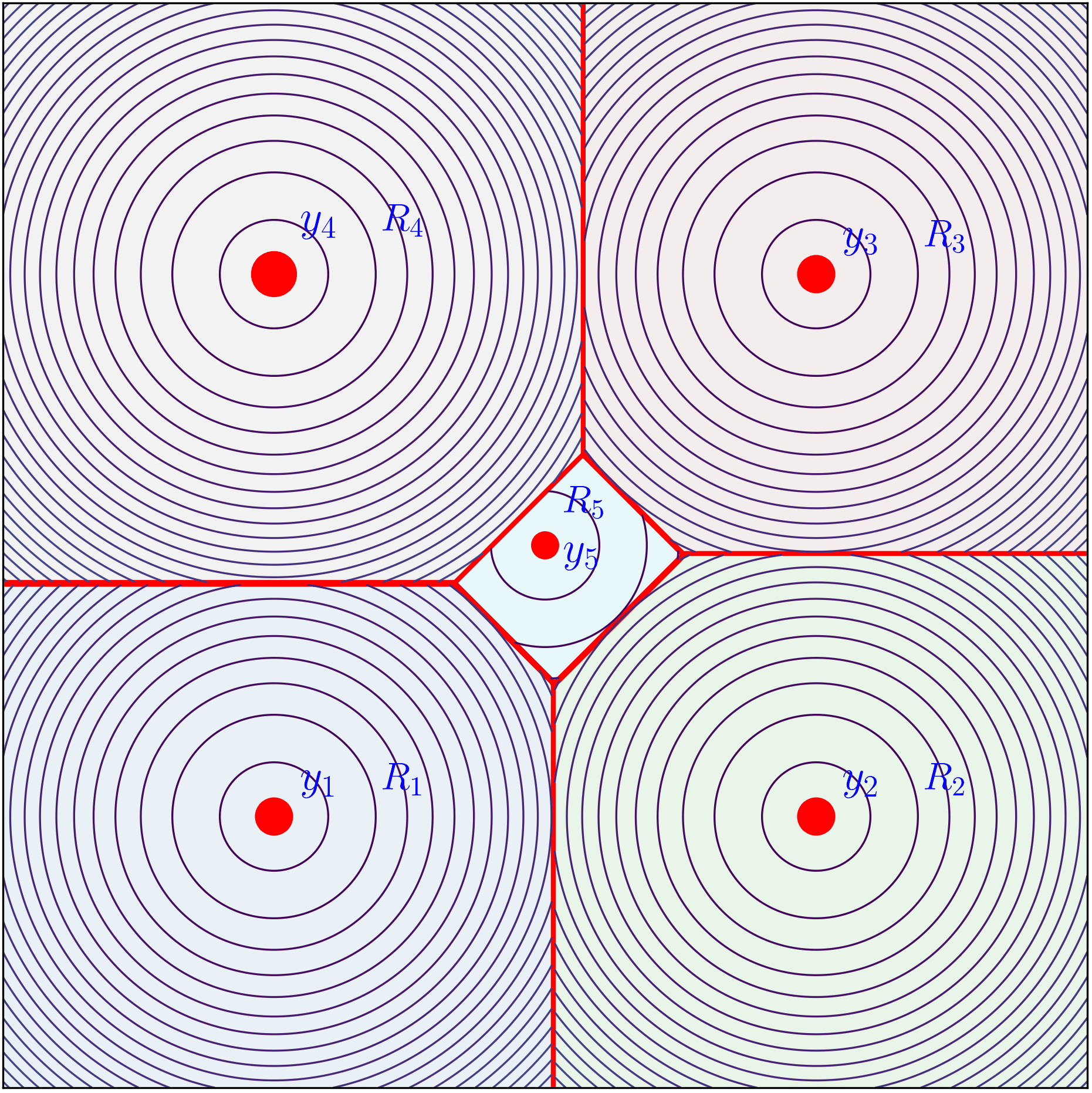}
                    & \includegraphics[width=0.32\linewidth]{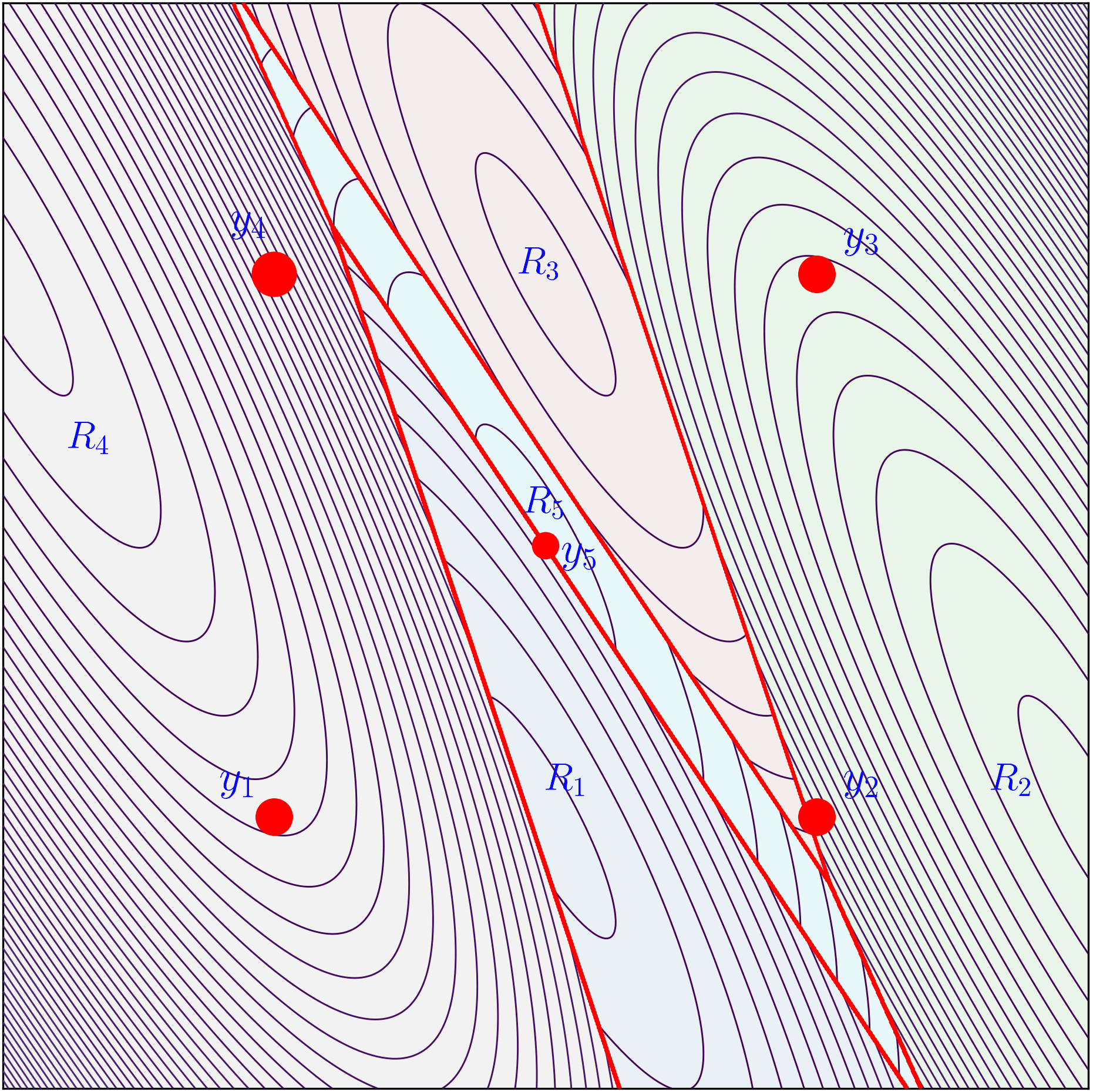}
                    & \includegraphics[width=0.32\linewidth]{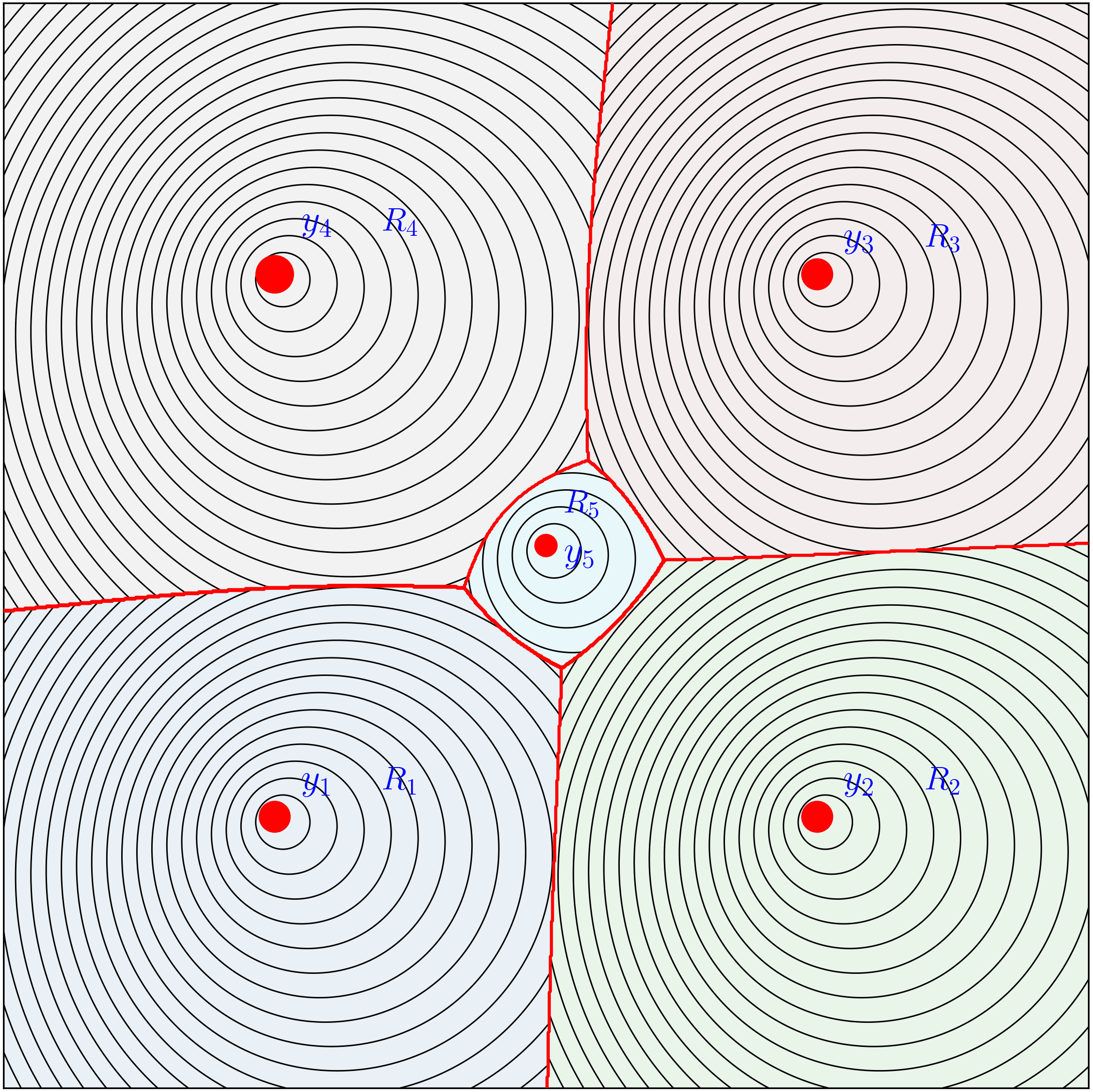}\\[-1pt]       \includegraphics[width=0.32\linewidth]{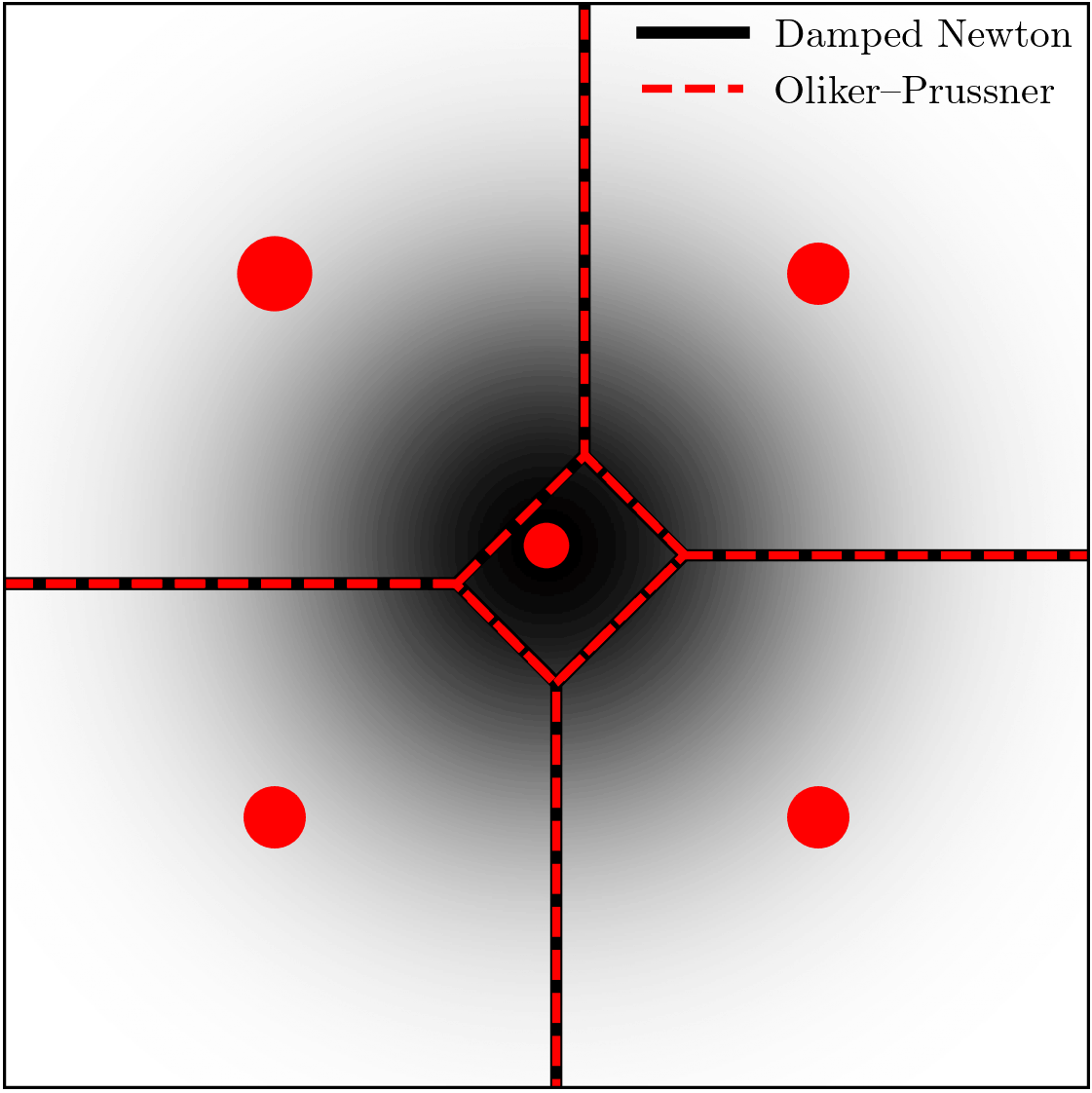}
                    & \includegraphics[width=0.32\linewidth]{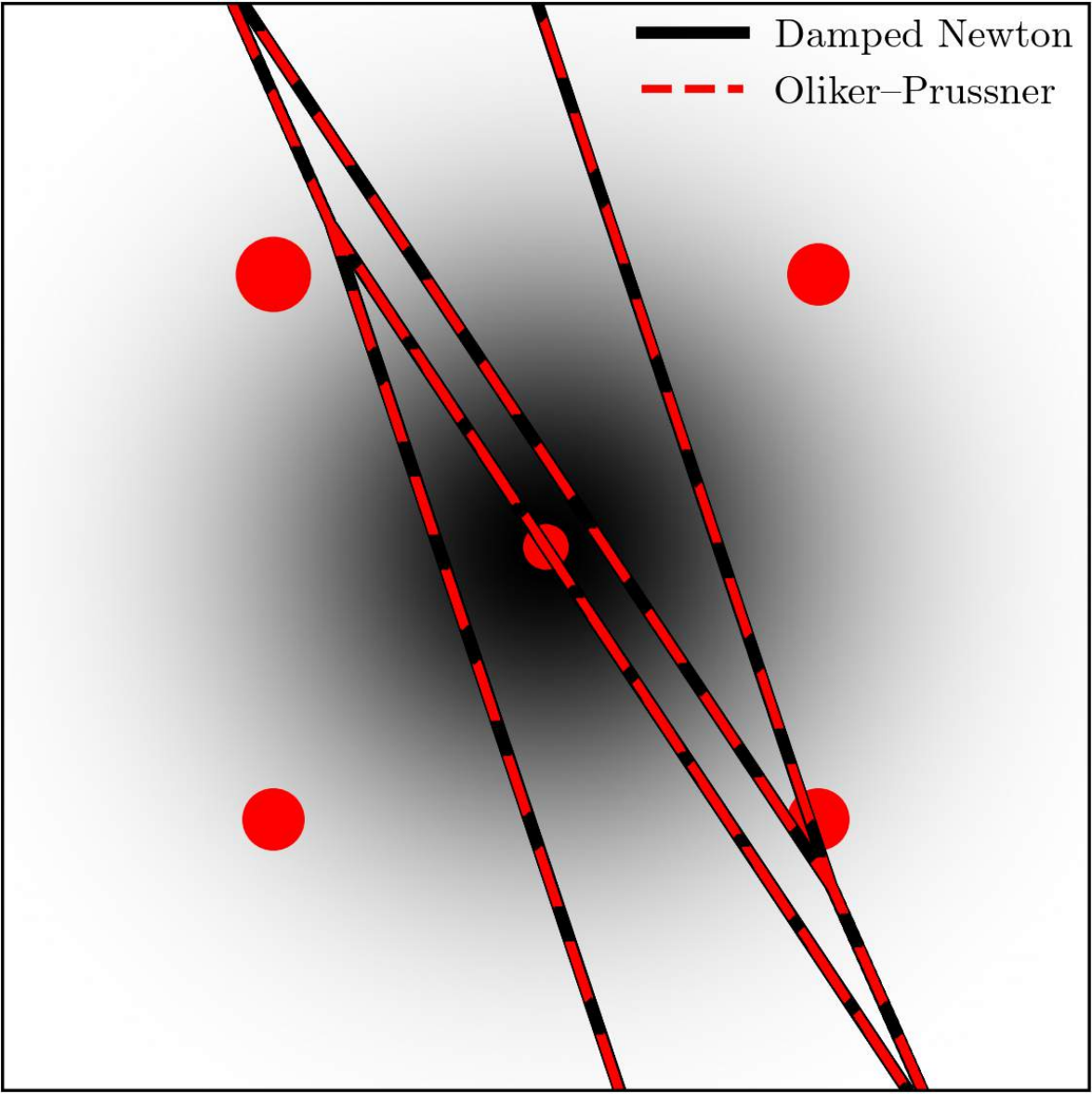}
                    & \includegraphics[width=0.32\linewidth]{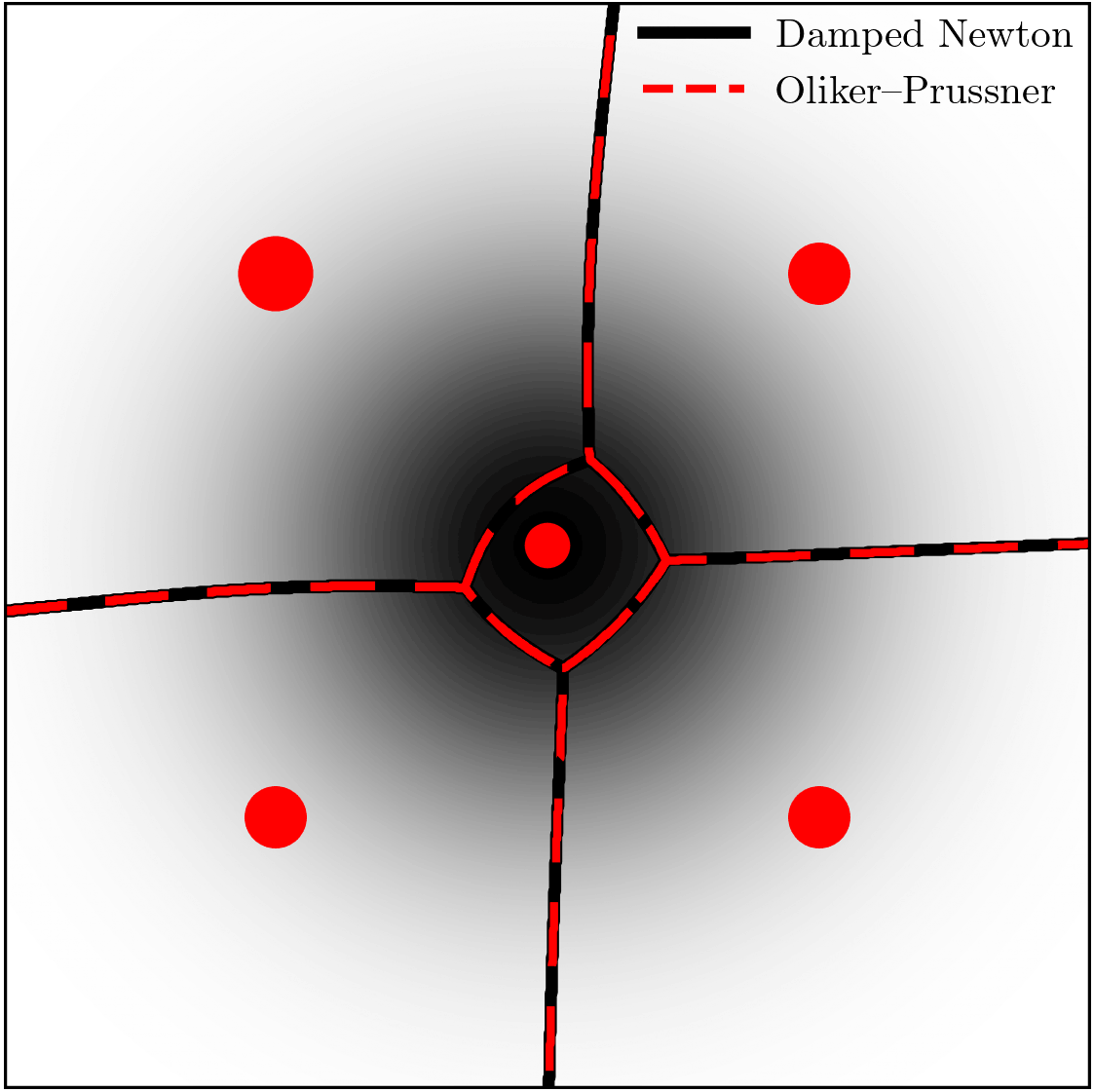}\\[-1pt]
                \includegraphics[width=0.32\linewidth]{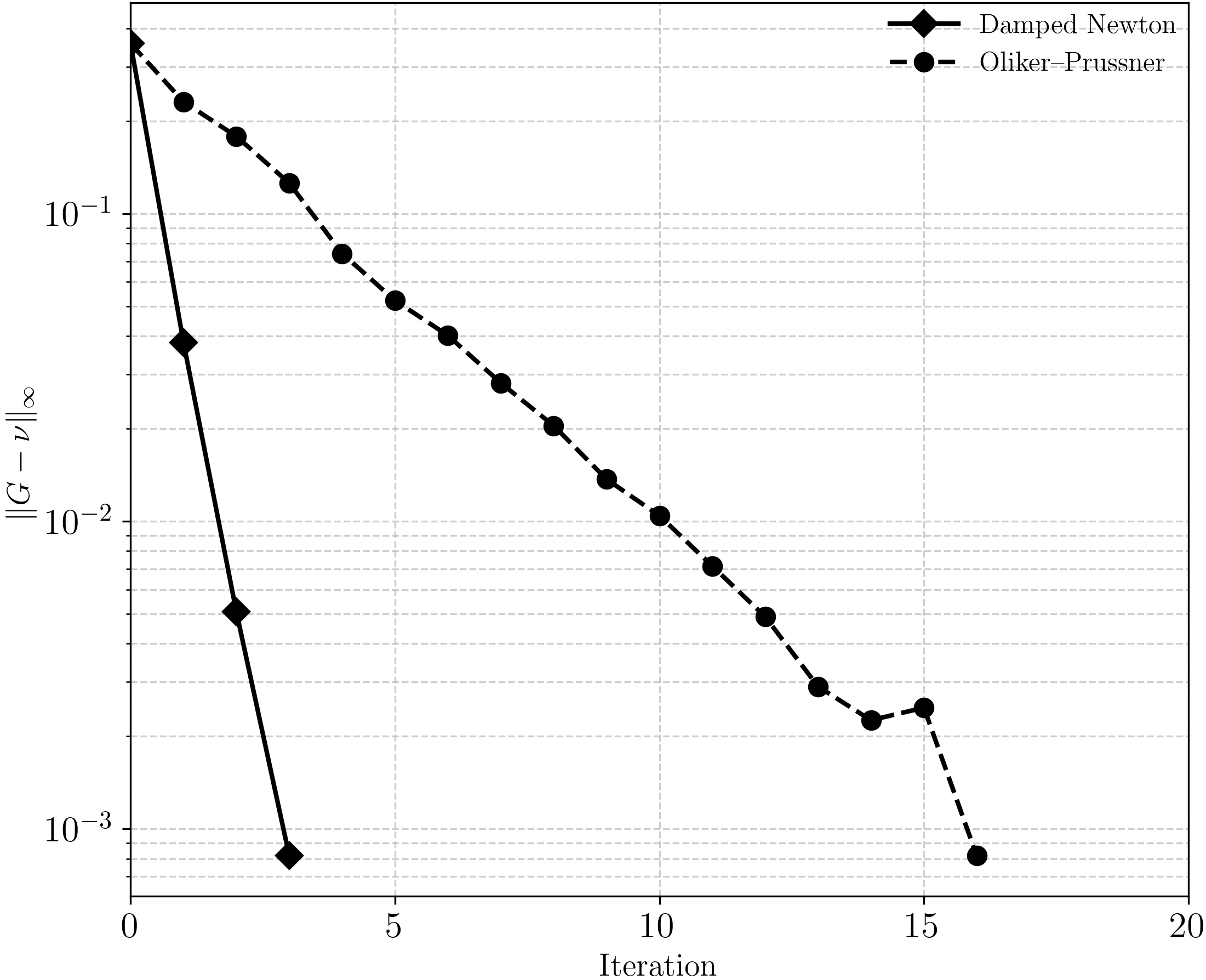}
                    & \includegraphics[width=0.32\linewidth]{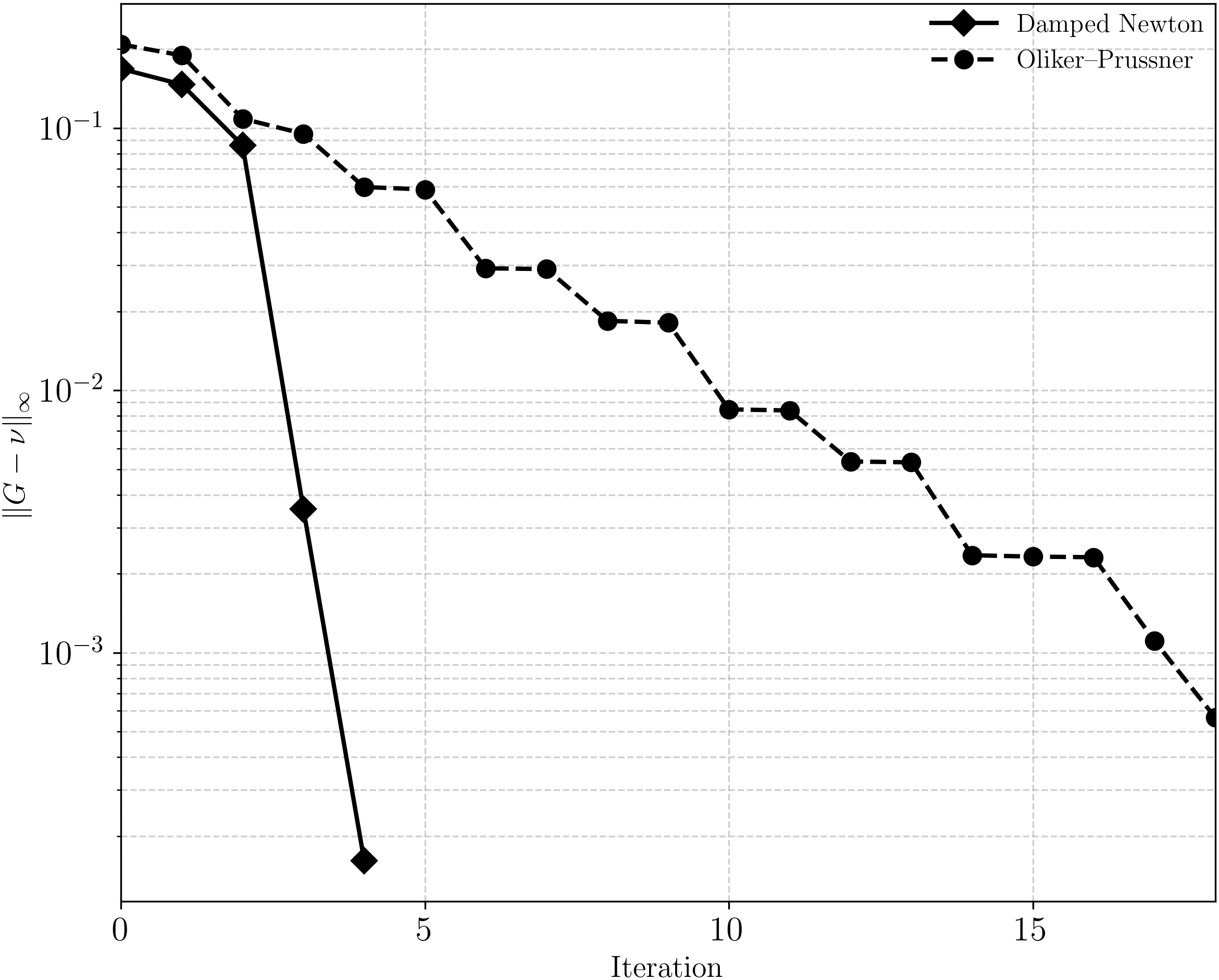}
                    & \includegraphics[width=0.32\linewidth]{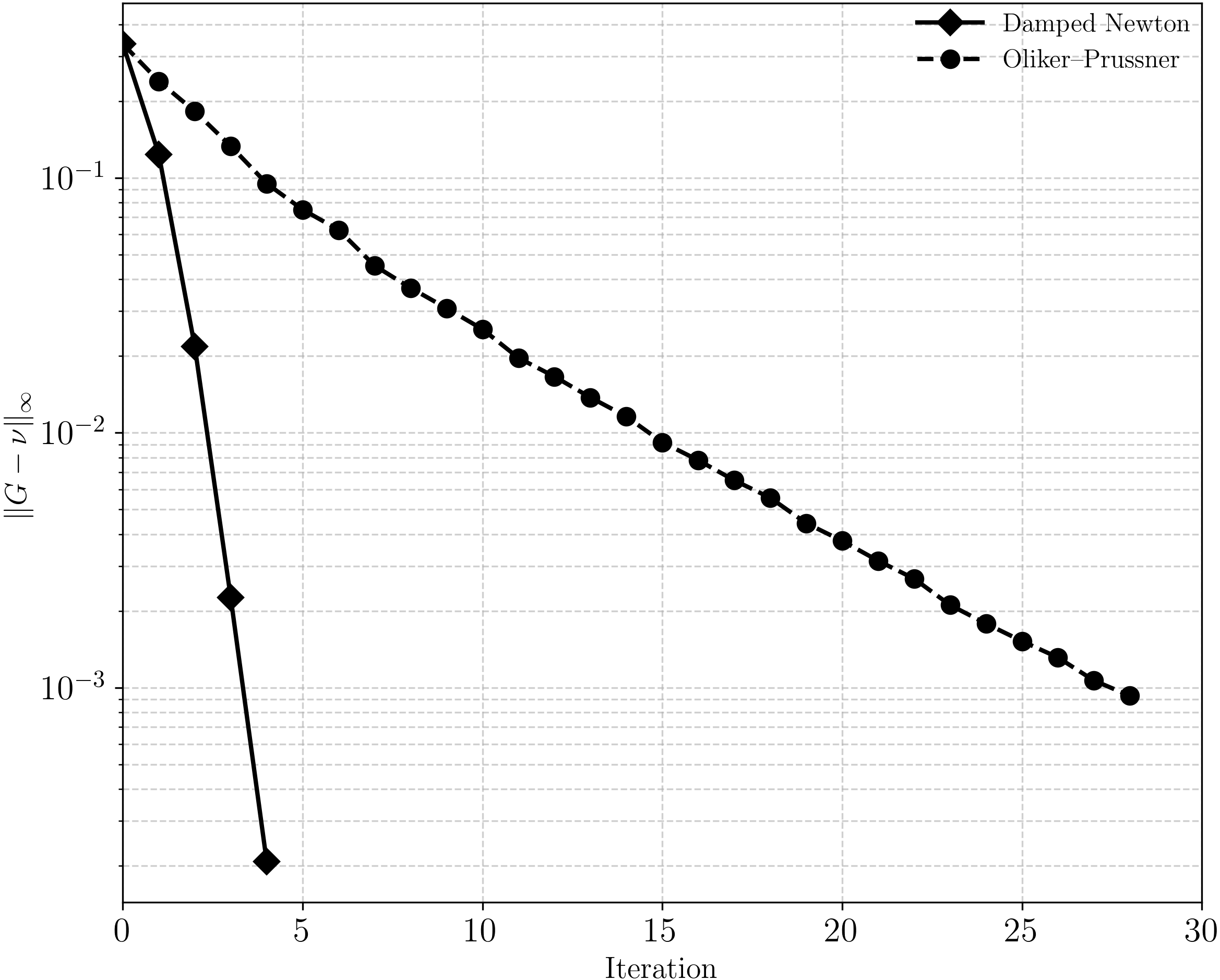}\\[-1pt]
\includegraphics[width=0.31\linewidth]{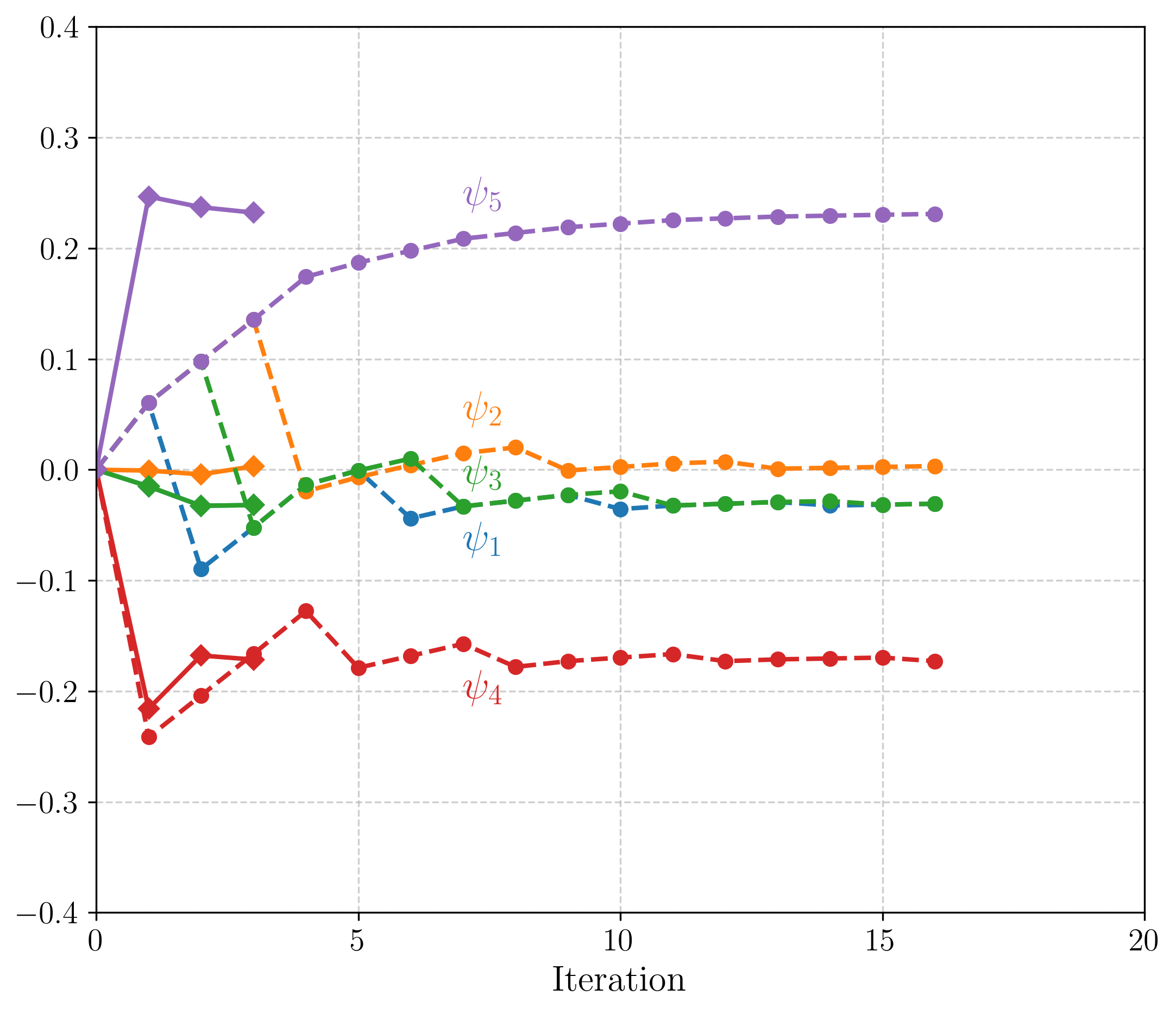}
                    & \includegraphics[width=0.32\linewidth]{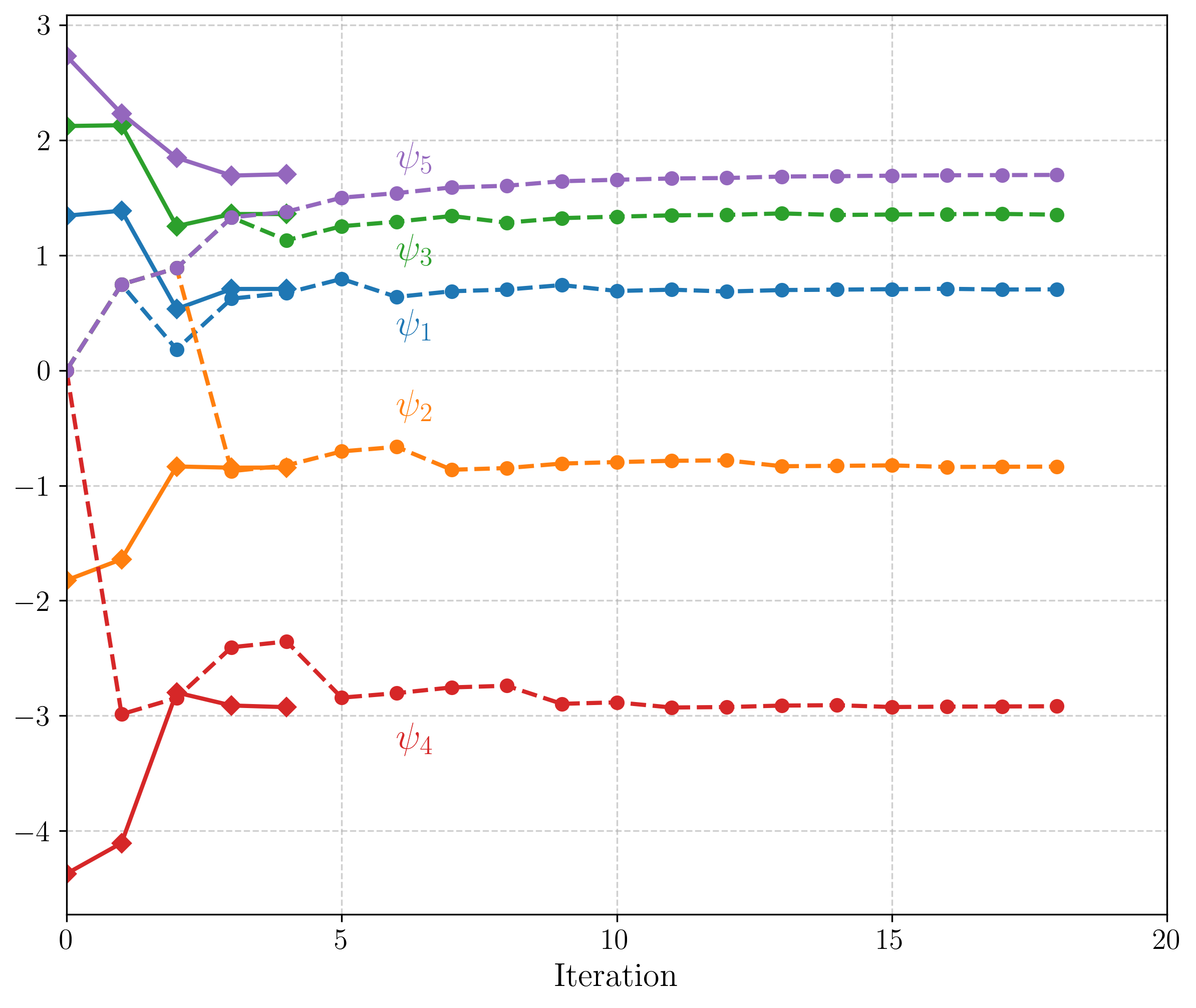}
                    & \includegraphics[width=0.32\linewidth]{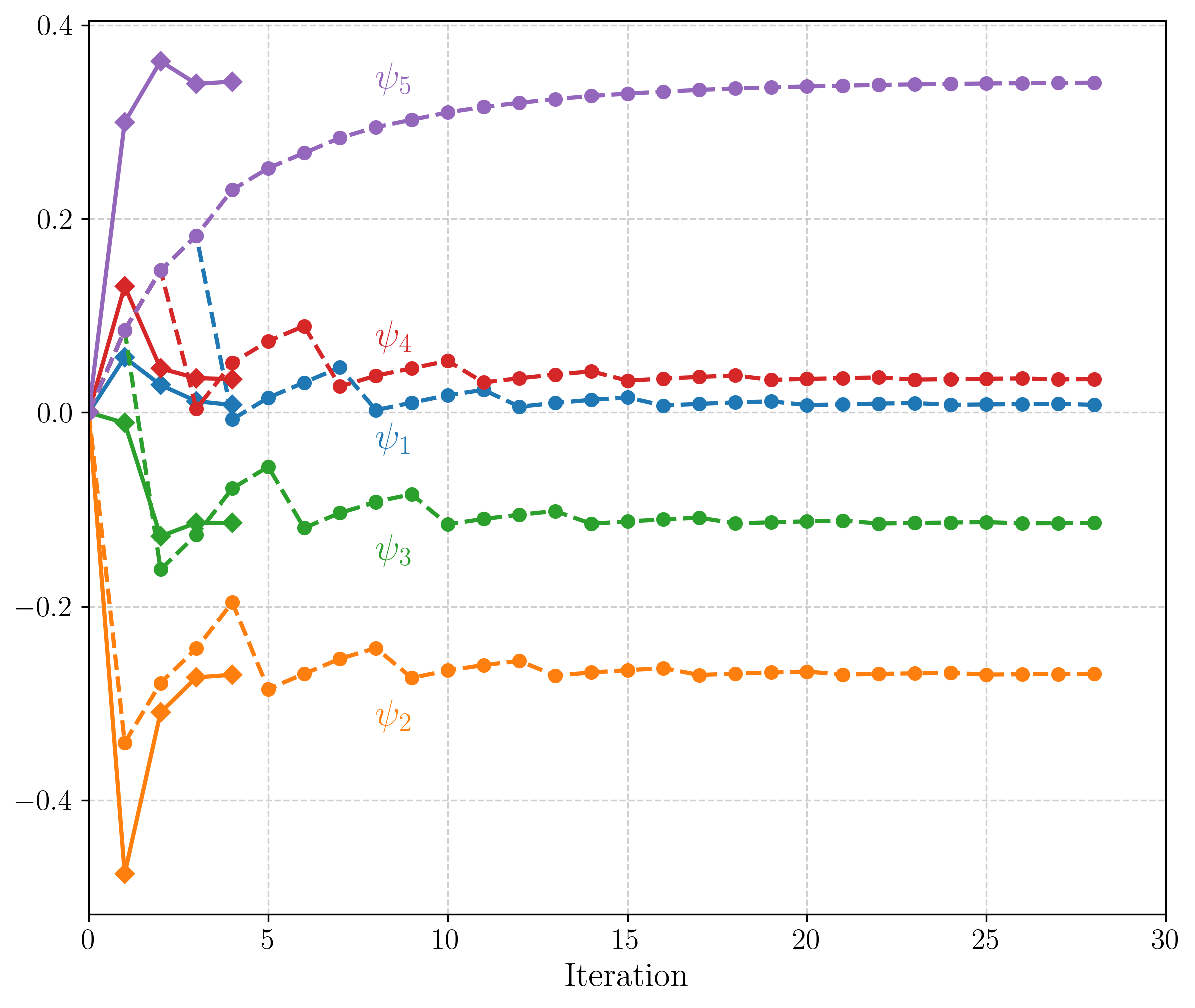}\\[-1pt]
                \end{tabular}%
\end{table*}

\noindent{\textbf{Control System Parameters for $c_{\mathrm{MinTime}}$.}} For the minimum time ground cost  \eqref{cZermeloMinTime}, we consider the setting in \cite[Sec. 6]{bakolas2010zermelo}, and fix (see Fig. \ref{fig:w})
\begin{align}
\bm{w}_t = \begin{cases}\bar{\bm{w}}+\bm{\rho} t, & 0 \leq t \leq \bar{t}, \\ \bar{\bm{w}}+\bm{\rho} \bar{t}, & t>\bar{t},\end{cases}
\label{wprofile}
\end{align}
where $\bar{\bm{w}}=(-0.3,0.2)^{\top}$, $\bm{\rho}=(0.05,-0.1)^{\top}$, $\bar{t}=4$, and $t\in[0,7]$. For these parameter values, the $\bm{w}_t$ in \eqref{wprofile} satisfies $\|\bm{w}_t\|_2\leq \|\bar{\bm{w}}\|_2 +\bar{t}\cdot\|\bm{\rho}\|_2 = (\sqrt{13}+\sqrt{20})/10 \approx 0.8078 < 1 \forall t\geq 0$. In this case, we numerically compute $c_{\mathrm{MinTime}}$ as the minimum of the real roots of a quartic and a quadratic equation whose coefficients are known functions of $\bar{\bm{w}},\bm{\rho},\bar{t}$, see \cite[Sec. 6]{bakolas2010zermelo}. For the above domain and problem data, that the projection condition holds was verified numerically. 


\noindent\textbf{Results and discussions.}
 Our numerical results are summarized in Table \ref{tab:picturegrid} for three ground costs: $c_{\mathrm{SqEuclidean}}$, $c_{\mathrm{MinEnergy}}$, $c_{\mathrm{MinTime}}$. For each of these ground costs, we solved \eqref{SDOT} using both damped Newton \cite{kitagawa2019convergence} and Oliker-Prussner \cite[Sec. 4.2]{merigot2021optimal} algorithms with stopping criterion $\|\bm{G}-\bm{\nu}\|_{\infty}<10^{-3}$. While the CLTs obtained from these two methods match well (2\textsuperscript{nd} row of Table \ref{tab:picturegrid}), the damped Newton method converged faster (3\textsuperscript{rd}-4\textsuperscript{th} rows of Table \ref{tab:picturegrid}). This faster convergence and runtimes (Table \ref{tab:comp_time}) for the damped Newton method are consistent with those reported in \cite{kitagawa2019convergence}. Also, the first two rows in Table \ref{tab:picturegrid} depict the convex polyhedral CLT for $c_{\mathrm{MinEnergy}}$ (Theorem \ref{thm:MinEnergyTwist}) and the nonconvex CLT for $c_{\mathrm{MinTime}}$ (Remark \ref{Remark:NonconvexCLT}). 



\begin{table}[t!]
\centering
\caption{{\small{Computational times for the CLTs in Sec. \ref{sec:Numerical}.
}}
}
\setlength{\tabcolsep}{5pt}        
\renewcommand{\arraystretch}{1.2}  
\begin{tabular}{l l l l}
\hline
Algorithm &  $c_{\mathrm{SqEuclidean}}$ &  $c_{\mathrm{MinEnergy}}$ &  $c_{\mathrm{MinTime}}$ \\
\hline\hline
Damped Newton      & 3.41 s & 4.64 s &  34.10 s\\
\hline
Oliker--Prussner   & 94.15 s & 105.55 s & 1295.00 s\\
\hline
\end{tabular}
\label{tab:comp_time}
\vspace*{-0.3in}
\end{table}




\bibliographystyle{IEEEtran}
\bibliography{references.bib}

\end{document}